\documentclass[10pt]{amsart}

\usepackage{a4}
\usepackage{amssymb}
\usepackage{amsmath}
\usepackage{amsfonts}
\usepackage{fancybox}
\usepackage{empheq}
\usepackage{mathtools}
\usepackage{calrsfs}

\usepackage[table]{xcolor}

\usepackage{color}
\usepackage{comment}

\usepackage[all,cmtip]{xy}

\usepackage{todonotes}
\usepackage{cleveref}

\usepackage{tikz}

\newtheorem{theorem}{Theorem}
\newtheorem{lemma}[theorem]{Lemma}
\newtheorem{corollary}[theorem]{Corollary}
\newtheorem{proposition}[theorem]{Proposition}
\newtheorem{example}[theorem]{Example}
\newtheorem{remark}[theorem]{Remark}
\newtheorem{definition}[theorem]{Definition}

\newcommand{\Ker}{\rm Ker}

\newcommand{\Aut}{\mathrm{Aut}}

\newcommand{\Z}{\mathbb{Z}}

\newcommand{\eq}[1]{\begin{equation}#1\end{equation}}
\newcommand{\B}[1]{\bar{#1}}
\newcommand{\Gal}{\mathrm{Gal}}

\title[HKG-covers]{A cohomological treatise of HKG-covers with applications to the Nottingham group}

\date{\today}

\author[A. Kontogeorgis]{Aristides Kontogeorgis}
\address{Department of Mathematics, National and Kapodistrian  University of Athens\\
Panepistimioupolis, 15784 Athens, Greece}
\email{kontogar@math.uoa.gr}
\urladdr{http://users.uoa.gr/~kontogar}

\author[I. Tsouknidas]{Ioannis Tsouknidas }
\address{Department of Mathematics, National and Kapodistrian University of Athens\\
Panepistimioupolis, 15784 Athens, Greece}
\email{iotsouknidas@math.uoa.gr}

\date \today

\makeatletter
\newcommand{\aprod}{\mathop{\operator@font \hbox{\Large$\ast$}}}
\makeatother

\begin{document}
\begin{abstract}
We characterize Harbater-Katz-Gabber curves in terms of a family of cohomology classes satisfying a compatibility condition. Our construction is applied to the description of finite subgroups of the Nottingham Group.
\end{abstract}

\maketitle


\section{Introduction}
In this article we use and extend results from previous work of the first author together with S. Karanikolopoulos \cite{Karanikolopoulos2013} on HKG-curves. We will work over an algebraically closed field $k$ of characteristic $p\geq 5$.

\begin{definition}
A Harbater-Katz-Gabber cover (HKG-cover for short) is a Galois cover $X_{\mathrm{HKG}}\rightarrow \mathbb{P}^1$, such that there are at most two  branched $k$-rational points $P_1,P_2 \in \mathbb{P}^1$, where $P_1$ is tamely ramified and $P_2$ is totally and wildly  ramified. All other geometric points of $\mathbb{P}^1$ remain unramified. In this article we are mainly interested in $p$-groups so our HKG-covers have a unique ramified point, which is totally and wildly ramified.  
\end{definition}

Work of Harbater \cite{harbater1980moduli} and of Katz and Gabber \cite{katz1986local} showed that any finite subgroup $G$ of $\Aut(k[[t]])$ can be associated with an HKG-curve $X$. More precisely, $G$ is the semi-direct product of a cyclic 
group of order prime to $p$ (the maximal tamely ramified quotient) by a normal $p$-subgroup (the wild inertia group). We are interested in the latter group, so from now on we will replace the initial group $G$ with the latter, finite $p$-subgroup of $\Aut(k[[t]])$. The HKG-curves play an important role in  the deformation theory of curves with automorphisms and to the celebrated proof of Oort conjecture, \cite{Obus12,ObusWewers, MR3591155, MR2441248, MR2919977, MR3194816}.

Working with the HKG-curve $X$ allows us to use several global tools like the genus, the $p$-rank of the Jacobian etc 
to the study of $k[[t]]$.
In this article we will employ the Weierstrass semigroup attached to the unique ramified point $P$, and we will use the results of \cite{Karanikolopoulos2013} on relating the structure of the Weierstrass semigroup to the jumps of the ramification filtration.  

More precisely, to the HKG-cover
there is a Weierstrass semigroup $H(P)$
attached to the unique wildly  ramified point $P$. 
An arithmetic semigroup, and in particular the Weierstrass semigroup, is always finitely generated, i.e. there are $\bar{m}_1,\ldots, \bar{m}_h \in \mathbb{N}$ such that
\[
H(P)=\mathbb{Z}_+ \bar{m}_1 + \cdots + \mathbb{Z}_+ \bar{m}_h.
\]
We will denote by $m_i$ the $i$-th element of $H(P)$, while $\bar{m}_i$ will denote the $i$-th generator of the semigroup.
For every element $m_i \in H(P)$ we will select a function $f_i$ with $(f_i)_\infty=m_i P$. Also each element $\bar{m}_i$ corresponds to some function 
$\bar{f}_i$ in the function field of the curve.
This selection is not unique and we will study later what happens by different choices either of  $f_i$ or  $\bar{f}_i$.
The ramification filtration gives rise to a series of subgroups of the group $G$, see eq. (\ref{ramfiltratseq}), which correspond to a sequence of subfields
$k(X/G)=F_0\subset \cdots \subset F_{s+1}=k(X)$
of the function field $F=k(X)$ of the curve $X$.
By the properties of the ramification filtration, each extension  $F_{i+1}/F_i$, $i=0,\ldots,s$ is abelian. We will see that 
$F_{i+1}=F_{i}(\bar{f}_i)$.  
 
One of the main results of this article is the 
classification  and description of  the Galois actions in HKG-covers  in terms of group cohomology. 
Assume that $X \rightarrow \mathbb{P}^1$ is an 
HKG-cover with a unique wildly ramified point $P\in X$. Consider the ring
of holomorphic functions outside the point $P$ 
\[
\mathbf{A}=\bigcup_{\nu=0}^\infty L( \nu P). 
\]
This ring is equipped with a valuation corresponding to $P$ and elements of $\mathbf{A}$ of valuation  smaller or equal than $\nu$, i.e. the Riemann-Roch space $L(\nu P)$, give rise to a vector space of finite dimension. Notice that these kind of rings are essential in the general definition of Drinfeld modules, see \cite[chap. 4]{Goss1996}.


Write $s$ for the index of the biggest $\bar{m}_i$, such that $\bar{m}_i<m$. Every intermediate extension $F_{i+1}/F_i$ is elementary abelian hence;
isomorphic to $(\mathbb{Z}/p\mathbb{Z})^{n_i}$.
Set $\mathbf{n}=(n_1,\ldots,n_s) \in \mathbb{N}^s$. In eq. (\ref{W-def}) we define the vector space  $k_{\mathbf{n},m}[\bar{f}_0,\ldots,\bar{f}_s]$  which will be considered as a $G$-module and prove that it is equal to $L\big( (m-1)P\big)$.

The polynomial ring $k[f_0]$ is the semigroup ring corresponding to the Weierstrass semigroup $\mathbb{N}$ of the projective line, which has bounded part the vector space $k_m[f_0]$ of polynomials of degree $\leq m$.  
The module $k_{\mathbf{n},m}[\bar{f}_0,\ldots,\bar{f}_s]$
plays a similar role for the more general setting of the Weierstrass semigroup of the HKG-cover.

The action of $G$ will be described by the following:
\begin{theorem}
\label{mainTH1}
The $G$-module structure of $k_{\mathbf{n},m}[\bar{f}_0,\ldots,\bar{f}_s]$
is described by a series of cohomology classes $\bar{C}_i \in H^1(\mathrm{Gal}(F_{i+1}/F_1),k_{\mathbf{n},\bar{m}_{i}}[\bar{f}_0,\bar{f}_1,\ldots,\bar{f}_{i-1}])$.
These classes  restricted to the elementary abelian group $\mathrm{Gal}(F_{i+1}/F_{i})$ 
define the additive polynomials $P_i(Y)$ 
which in turn describe the elementary abelian extensions $F_{i+1}/F_i$.
Moreover, the  additive polynomials $P_i$ define maps 
\[
H^1(G,k_{\mathbf{n},\bar{m}_i}[\bar{f}_0,\bar{f}_1,\ldots,\bar{f}_{i-1}])
 \longrightarrow 
H^1(G,k_{\mathbf{n},\bar{m}_i}[\bar{f}_0,\bar{f}_1,\ldots,\bar{f}_{i-1}])
\]
and the cocycles $\bar{C}_i$ are in the kernel of $P_i$,
that is 
\begin{equation}
\label{compat-eq}
P_i(\bar{C}_i)=0 \in H^1(G,k_{\mathbf{n},\bar{m}_i }[\bar{f}_0,\bar{f}_1,\ldots,\bar{f}_{i-1}]).
\end{equation}

Conversely every such series of elements $\bar{C}_i\in 
H^1(G,
k_{\mathbf{n},\bar{m}_i} [\bar{f}_0,\bar{f}_1,\ldots,\bar{f}_{i-1}])$, satisfying eq. (\ref{compat-eq}) defines in a unique way a HKG-cover. 
\end{theorem} 
\begin{proof}
We now sketch the ingredients of the proof. The precise proof will be given in the next sections of the article. 
Notice that the element $\bar{C}_i$ is the image of the class $\bar{f}_i \in L(\bar{m}_iP) /L\big( (\bar{m}_i-1) P\big)$ under the $\delta$ map in 
\[
\cdots \rightarrow 
H^0(G, L(\bar{m}_iP) /L\big( (\bar{m}_i-1) P\big)) 
\stackrel{\delta}{\longrightarrow} 
H^1 (G, L(\bar{m}_i P))
\rightarrow \cdots
\]
coming from the group cohomology long exact sequence corresponding to the short exact 
\[
\xymatrix{
  0 \ar[r] & 
  L\big(\bar{m}_i-1) P\big) \ar[r] &
  L(\bar{m}_i P) \ar[r] &
  L(\bar{m}_iP) /L\big( (\bar{m}_i-1) P\big) \ar[r] & 
  0
}
\]
Since the space $L(\bar{m}_iP) /L\big( (\bar{m}_i-1) P\big)$ is one dimensional the class $\bar{f}_i$ can be replaced by $\lambda \bar{f}_i$ for some $\lambda\in k^*$. 
In lemma \ref{coboundaries} we will explain further how the change of 
the semigroup generators gives rise to coboundaries.
  It turns out that changing $\bar{f}_i$ to $\lambda \bar{f}_i$ for $\lambda \in k^*$ changes  the additive polynomial $P_i$ to $\lambda^{p^{n_i-1}} P_i$ and the cocycle $\bar{C}_i$ to $\lambda \bar{C}_i$. This forces us to consider the projective space to the cohomology groups in order to obtain an independent of the generators description, see corollary 
  \ref{corProj}. 
 The definition of the additive polynomials $P_i$ and the compatibility condition is given in theorems \ref{thecomp} and  \ref{kernelCoho}.

The statement of the above theorem requires the selection of elements 
$\bar{f}_0,\ldots,\bar{f}_s$ of corresponding pole orders $\bar{m}_i$ generating the Weierstrass semigroup at the unique ramification point $P$.  The following corollary gives a description  independent of such a selection. 

\begin{corollary}
\label{corProj}
The HKG-cover $X$ can be completely described in terms of classes 
\[
[\bar{C}_i] \in \mathbb{P} 
H^1
\big(
\mathrm{Gal}(F_{i+1}/F_1),L(\bar{m}_iP)
\big)
\]
satisfying the compatibility conditions 
\[
[P_i] ([\bar{C}_i])=0 \in \mathbb{P}
H^1 \big(
\mathrm{Gal}(F_{i+1}/F_1),L(\bar{m}_iP)
\big),
\]
where $[P_i]$ is the map on projective spaces induced by the additive polynomials $P_i$. 
Once a selection of elements $\bar{f}_0,\ldots,\bar{f}_s$ is made all actions can be  expressed in terms of this expression. 
\end{corollary}

We will now indicate how we can construct the HKG-cover from the information of the compatible classes 
 $\bar{C}_i \in H^1(G,k_{\mathbf{n},m_{i}}[\bar{f}_0,\bar{f}_1,\ldots,\bar{f}_{i-1}])$. 
We argued that the additive polynomials can be constructed from the classes $\bar{C}_i$. The compatibility equation $P_i(\bar{C}_i)=0$ gives us that the cocycle representative $\bar{C}_i$ is a coboundary, that is, there is an element $D_i \in k_{\mathbf{n},m_{i}}[\bar{f}_0,\bar{f}_1,\ldots,\bar{f}_{i-1}])$ such that $P_i\big(\bar{C}_i(\sigma)\big)=(\sigma-1)D_i$ for all $\sigma\in G$. But then the element $\bar{f}_i$ satisfies the generalized Artin-Schreier extension 
\begin{equation}
\label{genAS1}
P_i(\bar{f}_i)=D_i,
\end{equation} see section \ref{tsouk8}. This essentially means that we can construct the HKG-cover step by step, adding in each step the generator $\bar{f}_i$ satisfying eq. (\ref{genAS1}).

\end{proof}
In fact the above theorem states that all the information for the HKG-cover is inside the sequence of compatible cohomology classes. This result  is similar to the cohomological interpretation of Kummer and Artin-Schreier-Witt extensions, see  
\cite[8.9-8.11]{Jac2}, \cite[chap. VI, sec. 1-2]{NeukirchCohoNumFields}. Of course 
Kummer and Artin-Schreier-Witt  extensions are abelian, while 
the HKG-extensions are solvable. This fits well with the  Shafarevich philosophy as expressed in \cite{MR1084537}.

As an application of the above result we give the following description of  finite $p$-subgroups of the Nottingham group:
\begin{theorem}
Let $G$ be a finite $p$-subgroup of the Nottingham group. 
There are elements $\bar{f}_0,\ldots,\bar{f}_s \in k[[t]]$ acted on by $G$ in terms of the cohomology classes $\bar{C}_i$ as described in theorem \ref{mainTH1}, 
and a local uniformizer $t'=(\bar{f}_s)^{1/m}$
so that 
\[
\sigma(t')= t'(1+\bar{C}_s(\sigma) (t')^m)^{-1/m} \in k[[t']]=k[[t]], \text{ for every } \sigma\in G.
\]
\end{theorem}
\begin{proof}
See theorem \ref{explicit-form} and the subsequent discussion. 
\end{proof}
   The above theorem is applied as follows: We start from a local action of a finite $p$-group $G$ on
 $k[[t]]$ and we construct an HKG-cover from it. From this HKG-cover we obtain the series of generators $\bar{f}_0,\ldots,\bar{f}_s\in k[[t]][t^{-1}]$ and we define the cohomology class $\bar{C}_s$, which in turn gives an explicit form of the action of $G$ on $k[[t']]$. Essentially we describe the conjugation class of $G \subset \Aut(k[[t]])$, since the group acting on
  $k[[t]]$ by 
  \[
  \sigma(t)=t 
  (1+\bar{C}_s(\sigma) t^m)^{-1/m},
   \] 
  is conjugate to our original action.

 We now describe the structure of this article.
In section \ref{sec:HKG-covers} we will introduce the representation and ramification filtration and their relation and we will also give a description of the Riemann-Roch space 
$L(m_i P)$ as polynomials of bounded degrees. 
Then we give a  cohomological  interpretation  of the action of the group $G$ and we also see  how the polynomials of each successive abelian extension can be recovered from this construction. In section \ref{sec:NotiGroups} we apply these tools in the problem of determining finite subgroups of the Nottingham group and in particular we give explicit forms of elements of order $p^h$. 
In the cyclic group case the cohomology group can be
expressed in terms of coinvariants of group action, see proposition \ref{22prop}. 
It seems that in recent years   interest on this problem has grown, see \cite{klopsch2000automorphisms}, \cite{Bleher2017}, \cite{lubin2011torsion}, \cite{DjurreMaster}. 



{\bf Acknowledgement:} 
 The authors would like to thank the anonymous referee for her/his comments on the improvement of the article.
This research is co-financed by Greece and the European Union (European Social Fund- ESF) through the Operational Programme ``Human Resources Development, Education and Lifelong Learning'' in the context of the project “Strengthening Human Resources Research Potential via Doctorate Research” (MIS-5000432), implemented by the State Scholarships Foundation (IKY).

\section{Generalities on HKG-covers}
\label{sec:HKG-covers}
\subsection{Ramification filtration}
 Let $X\rightarrow \mathbb{P}^1$ be a HKG-cover, that is Galois cover with Galois group a $p$-group $G$ fully ramified  over one point $P\in \mathbb{P}^1$.
In the associated HKG-curve $X$, the group $G$ will coincide with the inertia group of the curve at the unique ramified point, $G_T(P)=\{\sigma\in G(P):v_p(\sigma(t)-t)\geq 1\}$, where $t$ is a local uniformizer at $P$ and $v_P$ is the corresponding valuation. 
For more information on ramification filtration the reader is referred to \cite{serre2013local}.
 We define $G_i(P)$ to be the subgroup of $\sigma\in G(P)$ that acts trivially on $\mathcal{O}_p/m_P^{i+1}$, obtaining the following filtration;
\begin{equation}
\label{ramfiltratseq}
G_T(P)=G_0(P)=G_1(P)\supseteq G_2(P)\supseteq\cdots\supseteq\{1\}.
\end{equation}
Let us call an integer $i$ a \emph{jump of the ramification filtration} if $G_i(P)\gneqq G_{i+1}(P)$ and denote by
\eq{\label{tsouk4}
G_0(P)=G_1(P)=\cdots=G_{b_1}(P)\gneqq 
G_{b_1+1}(P)=\cdots=
G_{b_2}(P)\gneqq\cdots\gneqq G_{b_\mu}(P)\gneqq\{1\}
}
the filtration of the jumps, assuming that there are exactly $\mu$ jumps.

\subsection{The Weierstrass semigroup}\label{subsectionweierstrasssemigroup}
The Weierstrass semigroup $H(P)$ is the semigroup consisting of all pole numbers, i.e. $m\in \mathbb{N}$,  such that there is 
a function $f$ on $X$ with $(f)_\infty=mP$. 
For the Weierstrass semigroup  $H(P)$ we consider all pole numbers $m_i$ forming an increasing sequence
\[
0 = m_0 < \ldots < m_{r-1} < m_r,
\]
where $m_r$ is the first pole number not divisible by the characteristic. If $g\geq 2$ and $p\geq 5$ we can prove that $m_r\leq 2g-1$, see \cite[lemma 2.1]{kontogeorgis2008ramification}.

Let $F=k(X)$ be the function field of the HKG-curve $X$.
For every $m_i$,
$0\leq i\leq r$
in the Weierstrass semigroup we denote by 
$f_i\in F$  an element of $F$ that has a unique pole at $P$ of order $m_i$, i.e. $(f_i)_\infty=m_iP$. For each $i\in\{0,\dots, r\}$ the set $\{f_0,\dots,f_i\}$ forms a basis for the Riemann-Roch space $L(m_iP)$. The spaces 
\eq{
\label{L-flag}
k=L(m_0P) \subsetneq L(m_1 P) \subsetneq \cdots 
\subsetneq L(m_r P)
}
give rise to a natural flag of vector spaces corresponding to the Weierstrass semigroup.  
Notice that if $\mu$ is a pole number in $H(P)$ we have $\mu=m_{\dim L(\mu P)-1}$.

\subsection{Representation filtration}
For each $0\leq i \leq r$ we consider the representations
\begin{equation} \label{rho-def}
\rho_i: G_1(P) \rightarrow \mathrm{GL}(L(m_iP))
\end{equation}
which give rise to a decreasing sequence of groups 
\eq{\label{rep-fil}
G_1(P)=\mathrm{ker} \rho_0 \supseteq \mathrm{ker} \rho_1 \supseteq \mathrm{ker} \rho_2 \supseteq \ldots \supseteq \mathrm{ker} \rho_r =\{1\}.
}
Recall that $r$ is the index of $m_r$, the first pole number not divisible by $p$. 
In \cite{kontogeorgis2008ramification} the first author proved that $\rho_r$ is faithful {hence the last equality $\mathrm{ker}\rho_r=\{1\}$.

 We shall call the last filtration the \emph{representation filtration of $G$}.

\begin{definition}
An index $i$ is called a jump of the representation filtration if and only if $\mathrm{ker}\rho_i \gneqq \mathrm{ker} \rho_{i+1}$. 
\end{definition}

We will denote the jumps in the representation filtration by
\[
c_1 < c_2 < \ldots < c_{n-1} < c_n =r-1, 
\]
that is
\[
\mathrm{ker} \rho_{c_i} > \mathrm{ker}\rho_{c_i+1}. 
\]
The last equality $c_n=r-1$ is proved in \cite[rem. 9]{Karanikolopoulos2013}.
We have now a sequence of decreasing groups 
\eq{\label{tsouk6}
G_1(P) =\mathrm{ker} \rho_0 = \ldots = \mathrm{ker}\rho_{c_1} > \ldots \mathrm{ker} \rho_{c_{n-1}}> \mathrm{ker} \rho_{c_n} > \{1\}
}
which gives rise to the following sequence of extensions;
\eq{\label{tsouk3}
F^{G_1(P)}=F^{\Ker\rho_{c_1}}\subset F^{\Ker\rho_{c_2}}\subset\dots\subset F^{\Ker\rho_{c_n}}\subset F.
}




\subsection{A relation of the two filtrations in the case of HKG-covers}

In \cite{Karanikolopoulos2013} Karanikolopoulos and the first author related the filtrations defined in eq. (\ref{tsouk4}), (\ref{tsouk6}) and the Weierstrass semigroup in the following way;

\begin{theorem}\label{tsouk1}
We distinguish the following two cases:

\noindent$\bullet$
 If $G_1(P)> G_2(P)$ then the Weierstrass semigroup is minimally generated by 
$m_{c_i+1}=p^{h_i} \lambda_i$, $(\lambda_i,p)=1$,
$1\leq i \leq n$
  and the cover $F/F^{G_2(P)}$ is an HKG-cover as well. In this case $|G_2(P)|=m_1$.

\noindent$\bullet$
If $G_1(P)=G_2(P)$ then the Weierstrass semigroup is minimally generated
by $m_{c_i+1}=p^{h_i} \lambda_i$, $(\lambda,p)=1$, 
$1\leq i \leq n$
and by an extra generator  $p^h=|G_1(P)|$, which is different by all $m_{c_i+1}$ for all $1\leq i \leq n$. 

Especially when $X\to\mathbb{P}^1$ is an HKG-cover, the number of ramification jumps $\mu$ coincides with the number of representation jumps $n$, i.e. $n=\mu$.  The integers $\lambda_i$, which appear as factors of the integers $m_{c_i+1}$, $1\leq i \leq n$ are the jumps of the ramification filtration, i.e. $
\lambda_i=b_i$ and 
$G_{b_i}=\Ker\rho_{c_i}$
 for $2\leq i\leq n$.  Summing up we have the following options for the ramification filtration
\[
G_1(P)=\cdots = G_{\lambda_1} \gneqq 
G_{\lambda_1+1} = \cdots =G_{\lambda_2} \gneqq
G_{\lambda_2+1} = \cdots =
G_{\lambda_n} \gneqq \{1\}
\]
or
\[
G_1(P)>  G_2(P) = \cdots =G_{\lambda_1} \gneqq 
G_{\lambda_1+1} = \cdots= G_{\lambda_2} \gneqq
G_{\lambda_2+1} = \cdots =
G_{\lambda_n} \gneqq \{1\}
\]
\end{theorem}
\begin{proof}
See \cite[th. 13,th. 14]{Karanikolopoulos2013}.
\end{proof}
\begin{remark}
The reader should notice that $\Ker\rho_{c_1}=\Ker\rho_0=G_1(P)=G_{b_1}(P)$ by definition, hence $G_{b_i}=\Ker\rho_{c_i}$ for every $i\in\{1,\dots,n=\mu\}$.
\end{remark}
Theorem \ref{tsouk1} allows us to use the well known fact that the quotients $G_{b_i}/G_{b_{i+1}}$ are elementary abelian $p$-groups, hence the quotients $\Ker\rho_{c_i}/\Ker\rho_{c_{i+1}}$ are elementary abelian too, and the corresponding sequence of fields in (\ref{tsouk3}) is in fact, a sequence of elementary abelian $p$-group extensions.

In \cite[prop. 27]{Karanikolopoulos2013} the first author and 
S. Karanikolopoulos observed that for a $\sigma\in \Ker\rho_{c_i}-\Ker\rho_{c_{i+1}}$ the following hold;
\[\sigma(f_\nu)=f_\nu\text{ for all }\nu\leq c_i
\]
\[\sigma(f_{c_i+1})=f_{c_i+1}+C(\sigma)\text{ for some }C(\sigma)\in k^*.
\]
They also proved (prop. 20 \& rem. 21) that for each $i\in\{1,\dots,n\}$ we have $F^{\ker\rho_{c_{i+1}}}=F^{\ker\rho_{c_i}}(f_{c_i+1})$. 

In order to simplify notation we set $F_i:=F^{\ker\rho_{c_i}}$, $\bar{m}_i:=m_{c_i+1}$ and $\bar{f_i}:=f_{c_i+1}$, see also eq. (\ref{W-def}).

\begin{example}
 In the  Artin-Schreier extension $F=k(x)(y)$ 
where $y^p- y=x^m$ only the place $P=\infty$ is ramified with the following  ramification filtration: 
\[
\mathbb{Z}/p\mathbb{Z}=G_0=\cdots=G_m > \{1\},
\]
i.e. the first and unique ramification jump is at $m$, see \cite[prop. 3.7.8]{Stichtenothv2009}.
The representation filtration is given by 
\[
G_0=\mathrm{ker} \rho_0 =\cdots = \mathrm{ker} \rho_{m-1} > \{1\},
\]
that is, the first representation jump is at $c_1=m-1$ and $\bar{f}_1=f_{c_1+1}=y$, where $c_1=m-1$ and $c_1+1=m$. Thus $F=F_2=F_1(\bar{f_1})$, and $\bar{f}_0$ is the generator $x$ of the rational function field $k(x)$.
\end{example}

We will prove in section \ref{tsouk8} the following 
\begin{proposition}
\label{generatorsofGS}
 For a given $m \in H(P)$, in the case of HKG-covers we have
\[
L\big((m-1)P \big)=k_{\mathbf{n},m}[\bar{f}_0,\bar{f}_1,\ldots,\bar{f}_s],
\]
where
\begin{equation}
\label{W-def}
k_{\mathbf{n},m}[\bar{f}_0,\bar{f}_1,\ldots,\bar{f}_s]
=
\left\langle
\begin{array}{l}
\bar{f}_0^{a_0}\bar{f}_1^{a_1}\cdots \bar{f}_s^{a_s}: 
0 \leq a_i < p^{n_i} \text{ for all } 1\leq i \leq s,
\\
\text{ and }
\deg(\bar{f}_0^{a_0}\bar{f}_1^{a_1}\cdots \bar{f}_s^{a_s})=
\sum_{\nu=0}^s 
a_\nu \bar{m}_\nu< m
\end{array}
\right\rangle_k.
\end{equation} 
In the above equation $\deg(\bar{f}_i)$ is the pole order of $\bar{f}_i$ at $P$.
The integer $s$ is determined uniquely; it is the greatest index 
of $\bar{m}_i$ such that $\bar{m}_i< m$ holds. 
The quantity $\mathbf{n}=(n_1,\ldots,n_s) \in \mathbb{N}^s$ depends on the ramification filtration, specifically $n_i$ is the number of $\Z/p\Z$ components in each elementary abelian group $G_i/G_{i+1}$ obtained by quotients of the lower ramification filtration. 
\end{proposition}

\subsection{Groups acting on flags}
An automorphism of a curve act on all ``invariants'' of the curve including the Weierstrass semigroup of the unique ramified point. Usually this action on invariants provides useful information about the action. Unfortunately the action of the group $G$ on the semigroup $H(P)$ is trivial.
This is not the case when we move to the action to appropriate flags of vector spaces. 
More precisely we will consider flags 
of $k$-vector spaces
\[
\bar{V}: k=V_0 \subsetneq V_1  \subsetneq \cdots \subsetneq V_m \subsetneq \cdots
\]
where  $V_i=L(iP)$. 
We will say that a group  $G$ is acting on a flag $\bar{V}$, if there is a homomorphism 
\[
\rho:G \rightarrow \Aut(\bar{V}), 
\]
i.e. when $\rho(g)$ is an isomorphism such that $\rho(g) (V_i)=V_i$ for all $V_i$ in the flag. 


\begin{remark}
Since the representation $\rho_r$ is faithful it makes sense to consider the representation not on the whole flag but only up to $L(m_rP)$. The natural isomorphisms on this truncated flag are given by invertible upper triangular matrices. 
\end{remark}
  Recall that $s$ is the the greatest index of $\bar{m}_i$ such that $\bar{m}_i<m$. 
For every $1\leq i \leq s$  
and for every $1\leq j \leq r$
we have that 
\begin{align*}
\sigma(f_i) & =f_i+C_i(\sigma), \text{ where } C_i(\sigma) \in
 L \big( (m_i-1)P \big) 
 \\
\sigma(\bar{f}_i) & =\bar{f}_i+\bar{C}_i(\sigma), \text{ where } \bar{C}_i(\sigma) \in
 L \big( (\bar{m}_i-1)P \big) 
 .  
\end{align*}
Proposition \ref{generatorsofGS}, which will be proved in the next section, implies that if $\bar{f}_1,\ldots,\bar{f}_s$ are fixed, then  the values $\bar{C}_i$ for $1 \leq i \leq s$ determine the action completely. 

Also notice that for each $i\in \{1,\dots,r\}$, $f_i$ is a polynomial expression of the $\bar{f}_1,\ldots,\bar{f}_s$.
 By proposition \ref{generatorsofGS} we have $\bar{C}_i \in L \big( (\bar{m}_i-1)P \big)=
 k_{\mathbf{n},\bar{m}_i}[\bar{f}_0,\ldots,\bar{f}_{i-1}]$. 
The functions $\sigma \mapsto C_i(\sigma)$  and 
$\sigma \mapsto \bar{C}_i(\sigma)$ are 
cocycles, i.e.
\[
\bar{C}_i(\sigma\tau)=\bar{C}_i(\sigma)+ \sigma \bar{C}_i(\tau). 
\]
We plan to show that these cocycles define the action of $G$
on $X$, and in particular the finite subgroup of $\Aut(k[[t]])$.
\begin{remark}
The selection of the generators $\bar{f}_i$ for $0\leq i \leq s$ is not unique. Every element $a \in k_{\mathbf{n},m_i}[\bar{f}_0,\bar{f}_1,\ldots,\bar{f}_{i-1}]$ gives rise to a new generator $\bar{f}_i+a$. 

The new cocycle $\bar{C}_{i}'$ which is defined in terms of the generator $\bar{f}_i+a$ is given by 
\[
\sigma(\bar{f}_i+a)=\sigma(\bar{f}_i)+\sigma(a)
=\bar{f}_i +a + \bar{C}_i(\sigma)+\sigma(a)-a=
\bar{f}_i +a+
\bar{C}'_i(\sigma).
\] 
Therefore 
\[
\bar{C}'_i(\sigma)=\bar{C}_i(\sigma)+ (\sigma-1)a. 
\]
\end{remark}

Also instead of selecting the generator $\bar{f}_i$, which has pole order 
$\bar{m}_i$ at $P$ we can select $\lambda \bar{f}_i$ for any $\lambda\in k^*$. This change leads to cocycle $\lambda \bar{C}_i$. Therefore selecting the generator amounts to giving an element in the projective space
 \[
 \mathbb{P} 
 H^1\left(\frac{G}{\ker{\rho}_{i-1}},
k_{\mathbf{n},m_i}[\bar{f}_0,\bar{f}_1,\ldots,\bar{f}_{i-1}]
\right)
\]
This gives us the following
\begin{lemma}
\label{coboundaries}
The cocycles  $\bar{C}_i,\bar{C}_i'$ corresponding to 
different generators $\bar{f}_i,\bar{f}_i'$ with the same 
pole number $\bar{m}_i$, that is 
$\bar{f}'_i=\lambda \bar{f}_i +a$, $a \in k_{\mathbf{n},m_i}[\bar{f}_0,\bar{f}_1,\ldots,\bar{f}_{i-1}]$
satisfy the relation
\[
\bar{C}_i'(\sigma)=\lambda \bar{C}_i(\sigma) + (\sigma-1)\lambda a 
\]
and a generator free description of the action is determined by a series of classes $\tilde{C}_i$ in 
\begin{equation}
\label{cocycles-def}
\xymatrix{
 H^1
\left(
\frac{G}{\ker{\rho}_{i-1}},
k_{\mathbf{n},m_i}[\bar{f}_0,\bar{f}_1,\ldots,\bar{f}_{i-1}]
\right)
\ar@{^{(}->}[r]^{\qquad\mathrm{inf}} 
\ar[d]&
 H^1(G,k_{\mathbf{n},m_i}[\bar{f}_0,\bar{f}_1,\ldots,\bar{f}_{i-1}])
\ar[d] 
\\
\mathbb{P} H^1
\left(
\frac{G}{\ker{\rho}_{i-1}},
k_{\mathbf{n},m_i}[\bar{f}_0,\bar{f}_1,\ldots,\bar{f}_{i-1}]
\right)
\ar@{^{(}->}[r]^{\qquad\overline{\mathrm{inf}}}  &
\mathbb{P} H^1(G,k_{\mathbf{n},m_i}[\bar{f}_0,\bar{f}_1,\ldots,\bar{f}_{i-1}])
}
.
\end{equation}
\end{lemma}
These cocycles satisfy certain conditions which will be given in eq. (\ref{tsouk7}) and theorem \ref{kernelCoho}. 
The monomorphism $\mathrm{inf}$ is the inflation map in group cohomology, see \cite[II.2-3, p. 64]{Weiss}, while 
$\overline{\mathrm{inf}}[C]$ of the projective class $[C]$ of the cocycle $C$ is given by 
\[
\overline{\mathrm{inf}}[C]=[\mathrm{inf}(C)].
\] 
.
\begin{remark}
The vector space $k_{\mathbf{n},m_i}[\bar{f}_0,\bar{f}_1,\ldots,\bar{f}_{i-1}]$ has as base the space of monomials
$\bar{f}_0^{\nu_0} \bar{f}_1^{\nu_1}\ldots \bar{f}_{i-1}^{\nu_{i-1}}$, 
of degree smaller than $m$, where
$\nu_i < p^{n_i}$. The action on them can be described in terms of the binomial theorem, i.e. 
\begin{equation}
\label{exp-cf}
\bar{f}_0^{\nu_0} \bar{f}_1^{\nu_1}\!\cdots\! \bar{f}_{i-1}^{\nu_{i-1}}
\stackrel{\sigma}{\longrightarrow}
\bar{f}_0^{\nu_0}
\sum_{\mu_1}^{\nu_1}
\!\cdots\!
\sum_{\mu_{i-1}}^{\nu_{i-1}}
\binom{\mu_1}{\nu_1}
\!\cdots\!
\binom{\mu_{i-1}}{\nu_{i-1}}
\bar{f}_1^{\mu_1} \!\cdots\! \bar{f}_{i-1}^{\mu_{i-1}}
\bar{C}_1^{\nu_1-\mu_1} \!\cdots\! \bar{C}_{i-1}^{\nu_{i-1}-\mu_{i-1}}.
\end{equation}
\end{remark}

\subsection{Describing an HKG-cover as a sequence of Artin-Schreier extensions}\label{tsouk8}

It is known, see \cite{Garcia1991}, that every elementary abelian field extension $L/K$, with Galois group $(\Z/p\Z)^n$, is given as an Artin-Schreier extension of the form 
\[
L=K(y): \qquad y^{p^n}-y =b, \; b\in K.
\]
In our case, the elementary abelian field extension $F_{i+1}/F_i$ can be generated by an element $y\in F_{i+1}$ but this element might not be the semigroup generator $\B{f_i}$. 
We can give a description of the Artin-Schreier extension $F_{i+1}/F_i$ using a monic polynomial 
\[
A_i(X)=X^{p^{n_i}} + a_{n_i-1} X^{p^{n_i-1}} + \cdots + a_1 X^{p}+ a_0X -D_i,
\]
which can be computed in terms of the Moore determinant \cite{Goss1996}. Notice that this polynomial is an additive polynomial minus a constant term. 
Let $\{\sigma_1,\ldots,\sigma_{n_i}\}$ be a basis of the Galois group $\mathrm{Gal}(F_{i+1}/F_i)\cong (\Z/p\Z)^{n_i}$, seen as an $\mathbb{F}_p$-vector space, and let $w_1,\dots, w_{n_i}$ be elements of $k^*$ such that $\sigma_j(\bar{f_i})=\B{f}_i+w_j$. Let $W$ be the
 $\mathbb{F}_p$-subspace 
of $k$ spanned by the $w_j$, $j=1,\dots,n_i$.
We have $\dim_{\mathbb{F}_p} W=n_i$.


Let $P_{i}(X)=\prod_{a\in W}(X-a)$. Since every $w_i$ is an element of $k$, $\Gal(F_{i+1}/F_i)$ acts trivially on $P_i(X)$ and we consider the polynomial
$$A_i(X):=P_i(X)-P_i(\B{f_i}).$$
Notice that, for a $\sigma\in \mathrm{Gal}(F_{i+1}/F_i)$, we can write
$\sigma=\sigma_1^{\nu_1}\circ\cdots\circ\sigma_{n_i}^{\nu_{n_i}}$
and
\[
\sigma(\B{f_i}+a)=\B{f_i}+\nu_1w_1+\dots+\nu_{n_i}w_{n_i}+a, 
\text{ for all } a\in W\subset k.
\]  
This means that $P_i(\B{f_i})$ is $\Gal(F_{i+1}/F_i)$ invariant, i.e. belongs to $F_i$. Therefore, the polynomial $A_i(X)$ belongs to $F_i[X]$, is monic of degree $p^{n_i}=[F_{i+1}:F_i]$ and vanishes at $\B{f_i}$ hence it is the irreducible polynomial of $\B{f_i}$ over $F_i$.
The polynomial $P_i(X)$ is given by 
\begin{equation}
\label{moore-quot}
P_i(X)=
\frac{
\Delta(w_1,w_2,\ldots,w_{n_i},X)}
{ \Delta(w_1,w_2,\ldots,w_{n_i})
},
\end{equation}
where $\Delta(w_1,\dots,w_n)$ is the Moore determinant;
\[
\Delta(w_1,\dots,w_n)=
\det \begin{bmatrix}
w_1 & w_2 & \dots & w_n\\
w_1^p & w_2^p &\dots & w_n^p\\
\vdots &\vdots & & \vdots\\
w_1^{p^{n_i-1}} & w_2^{p^{n_i-1}} &\dots & w_{n_i}^{p^{n_i-1}}
\end{bmatrix}.
\]
It is an additive polynomial of the form
$$P_i(X)=X^{p^{n_i}}+a_{n_i-1}X^{p^{n_i-1}}+\dots+a_1X^p+a_0X,$$
where $a_i\in k\subset F_i$.
We have proved that the generator $\B{f_i}$ of the extension $F_{i+1}/F_i$ satisfies an equation of the form
\eq{\label{AS-rel}\B{f_i}^{p^{n_i}}+a_{n_i-1}\B{f_i}^{p^{n_i-1}}+\dots+a_1\B{f_i}^p+a_0\B{f_i}=D_i,}
for some $a_{n_i-1},\dots, a_0\in k$, $D_i=P_i(\B{f}_i)\in F_i$.
\begin{remark}
\label{remLmult}
Instead of $\bar{f}_i$ we can  use $\lambda \bar{f}_i$.  The additive polynomial corresponding to $\lambda \bar{f}_i$ is equal to $\lambda^{p^{n_i-1}} P_i(X)$, where $P_i(X)$ is the additive polynomial corresponding to $\bar{f}_i$. Indeed, when we change  $\bar{f}_i$ to $\lambda \bar{f}_i$  the $\mathbb{F}_p$-vector space $W$ is changed to $\lambda \cdot W$, that is the basis elements $w_i$ are changed to $\lambda w_i$. Hence, the  Moore determinant in the numerator of  eq. (\ref{moore-quot}) defining $P_i(\lambda X)$ is multiplied by
 $\lambda^{1+p+\cdots+p^{n_i-1}}$ while 
the denominator is multiplied by  $\lambda^{1+p+\cdots+p^{n_i-2}}$. Therefore $P_i(\lambda X)= \lambda^{p^{n_i-1}} P_i(X)$ follows.
\end{remark}
We have the  following:
\begin{theorem}
\label{thecomp}
The cocycles 
$\bar{C}_i \in H^1(\mathrm{Gal}(F_{i+1}/F_1),k_{\mathbf{n},
\bar{m}_{i}}[\bar{f}_0,\bar{f}_1,\ldots,\bar{f}_{i-1}])$, when restricted to the elementary abelian group
$\mathrm{Gal}(F_{i+1}/F_{i}) < \mathrm{Gal}(F_{i+1}/F_1)$
 describe fully 
the elementary abelian extension $F_{i+1}/F_i$ given by the equation 
\[
P_i(Y)=D_i.
\]
Moreover the element $D_i=P_i(\bar{f}_{i})$ is described by the additive polynomial $P_i(Y)$ 
and by the selection of $\bar{f}_i$. A different selection of $\bar{f}_i'$, i.e. $\bar{f}_i'=\lambda\bar{f}_i+a$, for some
 $a\in k_{\mathbf{n},m_{i}}[\bar{f}_0,\bar{f}_1,\ldots,\bar{f}_{i-1}]$,
$\lambda \in k^*$
 gives rise to the same polynomial $\lambda^{p^{n_i-1}}P_i$ and to a different $D_i'$
 given by $D_i'= \lambda^{p^{n_i-1}}D_i+ \lambda^{p^{n_i-1}}P_i(a)$. The two extensions $F_i(\bar{f}_i)$ and $F_i(\bar{f}_i')$ are equal. 
\end{theorem}
\begin{proof} 
The only part we didn't prove is the dependence of the additive polynomial to the selection of the generator $\bar{f}_i$.
 We have seen that changing $\bar{f}_i$ 
adds a coboundary to $\bar{C}_i$.  

 But when 
$\sigma$ belongs to $ \mathrm{Gal}(F_{i+1}/F_{i})$,  $\bar{C}_i(\sigma) $ belongs to $ k$, and $k$ admits the trivial action. Therefore, all coboundaries are zero and the result follows by lemma \ref{coboundaries}. 
\end{proof}

The additive polynomial $P_i(Y)$, which depends on the values of
 $\bar{C}_{i}(\sigma)$ with  $\sigma\in \mathrm{Gal}(F_{i+1}/F_{i})$ 
gives also compatibility conditions for the cocycle $\bar{C}_{i}$ on all elements of $\mathrm{Gal}(F_{i+1}/F_1)$. 
Namely, by application of $\sigma$ to eq. (\ref{AS-rel}) we obtain the following
\eq{\label{tsouk7}
\boxed{
P_i (\B{C}_i(\sigma))=(\sigma-1)D_i \text{ for all } \sigma\in \mathrm{Gal}(F_i/F_1).
}
}
So if $\sigma$ keeps $D_i$ invariant, for instance when  $\sigma \in \mathrm{Gal}(F/F_i)$, then $\B{C}_i(\sigma) \in \mathbb{F}_{p^n} \subset k$.

Equation (\ref{tsouk7}) is essentially a relation among the cocycles $\bar{C}_i(\sigma)$ and $\bar{C}_\nu(\sigma)$ for $\nu<i$. Indeed, the element $D_i \in k_{\mathbf{n},m_{i}}[\bar{f}_0,\bar{f}_1,\ldots,\bar{f}_{i-1}]$ is a polynomial expression on the elements $\bar{f}_0,\ldots,\bar{f}_{i-1}$, and the action is given in terms of the elements $\bar{C}_\nu(\sigma)$ for $\nu<i$
 and $\bar{f}_i$ as given in eq. (\ref{exp-cf}).


\begin{lemma}
\label{actA}
An additive polynomial  $P\in k[Y]$ defines a map
\begin{align}
H^1(G,k_{\mathbf{n},m_{i}}[\bar{f}_0,\bar{f}_1,\ldots,\bar{f}_{i-1}])
& \longrightarrow 
H^1(G,
k_{\mathbf{n},m_{i}}[\bar{f}_0,\bar{f}_1,\ldots,\bar{f}_{i-1}]
)
 \label{P-map}
\\
d \longmapsto P(d), \nonumber
\end{align}
\end{lemma}
\begin{proof}
Notice first that elements in the space $L(\nu P)$, for some $\nu\in \mathbb{N}$, can be multiplied as elements of the ring $\mathbf{A}$, so a polynomial expression $P(d)$ of a cocycle $d$ makes sense. 
One has to be careful since the multiplication of two elements in $L(\nu P)$, is not in general an element of $L(\nu P)$, since it can have a pole order greater than $\nu$. 
Therefore the value $P(d)$ is an element in $L(\mu P)$
for some $\mu \in \mathbb{N}$ for   big enough $\mu$.
However notice that  eq. (\ref{tsouk7}) implies that $P(\bar{C}_i(\sigma)) \in k_{\mathbf{n},\bar{m}_{i}}[\bar{f}_0,\bar{f}_1,\ldots,\bar{f}_{i-1}]$
so that $P_i(\bar{C}_i) \in H^1(G,k_{\mathbf{n},\bar{m}_{i}}[\bar{f}_0,\bar{f}_1,\ldots,\bar{f}_{i-1}])$.

Finally  observe now that if $d$ is a cocycle, i.e. $d(\sigma \tau)=d(\sigma)+\sigma d(\tau)$, then 
\[
P(d(\sigma \tau	))=P(d(\sigma)+\sigma d(\tau))=
P(d(\sigma))+P(\sigma d(\tau))=
P(d(\sigma))+\sigma P(d(\tau)).
\]
On the other hand if $d(\sigma)=(\sigma-1)b$ is a coboundary, 
then 
\[
P(d(\sigma))=P\big( (\sigma-1)b\big)=(\sigma-1)P(b)
\] 
is a coboundary as well.
\end{proof}
This allows us to give a cohomological interpretation of  eq. (\ref{tsouk7}):
\begin{theorem}
\label{kernelCoho}
The cocycles $\bar{C}_i$ given in eq. (\ref{cocycles-def}) are in the kernel of the map $P_i$ acting on cohomology as defined in lemma \ref{actA}. The corresponding element $D_i$ is then the 
element expressing $P(C_i)$ as a coboundary. The elementary abelian extension is determined by a series of cocycles
 $\bar{C}_i\in H^1(\mathrm{Gal}(F_{i+1}/F_i), k_{\mathbf{n},\bar{m}_i}[\bar{f}_0,\bar{f}_1,\ldots,\bar{f}_{i-1}])$, which define a series of additive polynomials $P_i$ and extend to cocycles 
 in $\bar{C}_i \in H^1(\mathrm{Gal}(F_{i+1}/F_1), k_{\mathbf{n},\bar{m}_i}[\bar{f}_0,\bar{f}_1,\ldots,\bar{f}_{i-1}])$ so that each $\bar{C}_i$ is in the kernel of $P_i$. 
\end{theorem}

\begin{remark}
In remark \ref{remLmult} we have seen that by changing the generator $\bar{f}_0$ to $\lambda \bar{f}_0$ the additive polynomial is changed from $P_i$ to $\lambda^{p^{n_i-1}} P_i$. The corresponding map 
\[
\mathbb{P} 
H^1(G,k_{\mathbf{n},m_{i}}[\bar{f}_0,\bar{f}_1,\ldots,\bar{f}_{i-1}])
 \longrightarrow 
\mathbb{P}
H^1(G,
k_{\mathbf{n},m_{i}}[\bar{f}_0,\bar{f}_1,\ldots,\bar{f}_{i-1}]
)
\]
is not affected. 
\end{remark}

\section{Nottingham groups}
\label{sec:NotiGroups}
An automorphism $\sigma$ of the complete local 
 algebra
 $k[[t]]$ is determined by the image $\sigma(t)$ of $t$, 
where $\sigma(t) =\sum_{i=1}^\infty a_it^i\in k[[t]]$. We consider the subgroup of normalised automorphisms that is, automorphisms of the form
$$\sigma:t\mapsto t+\sum_{i=2}^\infty a_it^i.$$
S. Jennings \cite{jennings1954substitution} proved that the set of latter automorphisms forms a group under substitution, denoted by $\mathcal{N}(k)$, called the Nottingham group. This group has many interesting properties, for instance R. Camina proved in \cite{camina1997subgroups} that every countably based pro-$p$ group can be embedded, as a closed subgroup, in the Nottingham group. We refer the reader to \cite{camina2000nottingham} for more information regarding $\mathcal{N}(k)$. We would like to provide an explicit way to describe the elements of $\mathcal{N}(k)$.
It is proved in \cite[prop. 1.2]{klopsch2000automorphisms} and \cite[sec. 4, th. 2.2]{lubin2011torsion}, that each automorphism of order $p$ is conjugate to the automorphism given by 
\begin{equation} \label{explicit-form}
t \mapsto t (1+ct^m)^{-1/m} = 
t \left(\sum_{\nu=0}^\infty 
\binom{-1/m}{\nu} c^\nu t^{\nu m}
\right)
\end{equation}
for some $c\in k^\times$ and some positive integer $m$ prime to $p$.

In \cite{Bleher2017} F. Bleher, T. Chinburg, B. Poonen and P. Symonds, studied 
the extension $L/k(t)$, where $L:=k(\{\sigma(t): \sigma \in G \})$,  
where $G$ is a finite subgroup of $\mathrm{Aut} k[[t]]$. 
Notice here that each automorphism of order $p^n$ is conjugate to $t\mapsto \sigma(t)$, where $\sigma(t) \in k[[t]]$ is algebraic over $k(t)$.  Also in \cite{Bleher2017} 
the notion of almost rational automorphism is defined: an automorphism $\sigma \in \Aut(k[[t]])$ is called almost rational if the extension  $L/k(t)$ is  Artin-Schreier.

The rational function field $k(t)$, despite its simple form, is not natural with respect to the group $G$ acting on the HKG-cover. 
For example the determination of the algebraic extension $L/k(t)$ 
and the group of the normal closure seems very difficult. 
 
Here we plan to give another generalization, by using the fact that the ``natural'' rational function field with respect to the Katz-Gabber cover is $X^{G_1}$ and not $k(t)$. 

In \cite[p. 473]{kontogeorgis2008ramification} the first author proposed the following explicit form for an automorphism of an HKG-cover of order $p^n$;
\[
\sigma(t)=t \left(1+\sum_{i=1}^r c_i(\sigma) u_i t^{m-m_i} \right)^{-1/m}, 
\]
where $m$ is the first pole number which is not divisible by the characteristic $p$, $u_i/t^{m_i}$ for $1 \leq i \leq r$ are functions in $L(mP)$ ($u_i$ is a unit) and $1/t^m$ is the function corresponding to $m$ ($t$ being the local uniformizer). In the latter function the unit is absorbed  by Hensel's lemma. 
%
%
%
\subsection{A canonical selection of uniformizer}
In an attempt to describe in explicit form automorphisms 
of $k[[t]]$ let us quote here some results from 
\cite{kontogeorgis2008ramification}. We will work with the corresponding HKG-cover $X \stackrel{G}{\longrightarrow} \mathbb{P}^1$ corresponding to a finite subgroup $G \subset \Aut (k[[t]])$. Again let $m_r$ denote the first pole number not divisible by the characteristic and $f_i$, $i=1,\ldots,\dim L(m_rP)=r$ a basis for the space $L(m_rP)$, such that 
\begin{equation} \label{poleSeq}
(f_i)_\infty=m_i.
\end{equation}
As we have seen 
this basis is not unique but eq.(\ref{poleSeq}) implies that if the element $f_i$ is selected, then $f_i'= \lambda_i f_i +a_i$, where
$a_i\in L\big( (m_i-1)P\big)$  is also a basis element of valuation $m_r$.

This means that the base change we will consider, corresponds to invertible upper triangular matrices, i.e. to linear maps which keep the flag of the vector spaces $L(m_i P)$. 

Recall that $m=m_r$ is the first pole number not divisible by $p$.
Let us focus on the element $f_r$. This element is of the form $f_r=u_m/t^m$, where $u_m$ is a unit. Since $(m,p)=1$ we know by Hensel's lemma that $u_m$ is an $m$-th power so by a change of uniformizer we can assume that $f_r=1/t^m$. 
When changing from a uniformizer $t$ to a uniformizer
$t'=\phi(t)=tu(t)$ ($u(t)$ is a unit in $k[[t]]$), 
the automorphism $\sigma\in k[[t]]$ expressed as an element in $k[[t']]$  is a conjugate of the initial automorphism, i.e. $ \phi \sigma \phi^{-1}$. By selecting the  canonical uniformizer with respect to $f_r$ we see that the expression of an arbitrary  $\sigma$ can take a simpler representation after conjugation.
Also this result is in accordance with (and can be seen as a generalization of) the result of Klopsch and Lubin, \cite{klopsch2000automorphisms}, \cite{lubin2011torsion}.
The selection of uniformizer $t=t_{f_r}$ is unique once $f_r$ is selected. 

\begin{definition}
We will call the uniformizer $t_{f_r}=f_r^{-1/m}$ the {\em canonical uniformizer} corresponding to $f_r$.
\end{definition}
What happens if we change
the function $f_r$ to
 $f'_r=f_r+a$,
 where $a \in L\big( (m-1) P\big)$?  Then $a=u/t^\mu$, with $0 \leq \mu< m$ and in this case the new uniformizer is given by 
\[
t_{f_r'}=\left(f_r+\frac{u}{t^\mu} \right)^{-1/m}
= t
\left( 1 +u t^{m-\mu} \right)^{-1/m}
=t \left(1+ a t^m \right)^{-1/m}. 
\]
Keep in mind that the set of uniformizers for the local ring $k[[t]]$ equals to $t u(t)$, where $u$ is a unit of the ring $k[[t]]$.  

Let $\bar{m}_1, \ldots,\bar{m}_s$ be the generators of the Weierstrass semigroup $H(P)$. 
These elements correspond to a successive sequence of 
function fields $F_i=F_{i-1}(\bar{f}_{i-1})$ so that 
$v(\bar{f}_{i-1})= p^{|\mathrm{Gal}(F/F_i)|} \lambda_{h-1}=\bar{m}_i$. It is not clear that $\bar{m}_i \geq \bar{m}_j$ for $j<i$. However if for some $j$ we have 
$\bar{m}_j < \bar{m}_i$ for some $i<j$ then 
\[
\sigma(\bar{f}_j) =\sigma(\bar{f}_j)+ \bar{C}_j(\sigma),
\text{ where } \bar{C}_j \in k[\bar{f}_0,\ldots,
\widehat{\bar{f}_i},\ldots,\bar{f}_{j-1}],
\]
that is, $\bar{f}_i$ does not appear in any term of the polynomial expression of $\bar{C}_j(\sigma)$, for all $\sigma \in G$. This means that we can generate an HKG-cover with corresponding function field generated by fewer elements than the initial one.

If we assume that among all  HKG-covers which correspond to a local action of $G$ on $k[[t]]$ we select  one whose function field is minimally generated then 
$\bar{m}_1 < \bar{m}_2 < \ldots < \bar{m}_s$.

\begin{lemma}
\label{index-m}
Let $m=m_r$ be the  first pole number not divisible by the characteristic $p$. Then $m=\bar{m}_s$, that is the pole number corresponding to the last generator $\bar{f}_s$. 
\end{lemma}

\begin{proof}
It is clear that not all pole numbers are divisible by $p$ since $m\in H(P)$, $p\nmid m$. 
So at least one generator must be prime to $p$. On the 
other hand $F_i=F_{i-1}(\bar{f}_{i-1})$, thus the pole numbers $\bar{m}_i$ of elements $\bar{f}_i$ for $i<s$ are divisible by $p$, see also \cite[eq. (6)]{Karanikolopoulos2013}. Therefore only the last generator can be not divisible by $p$.  
\end{proof}

\begin{theorem}
\label{explicitForm}
Let $\bar{C}_s \in H^1(G,k_{\mathbf{n},m}[\bar{f}_0,\bar{f}_1,\ldots,\bar{f}_{s-1}])$ be the cocycle 
corresponding to $m=m_s$, where $m$ is the first pole number not divisible by $p$, see lemma \ref{index-m}.
We choose as uniformizer  the canonical uniformizer  $t=\bar{f}_s^{-1/m}$.
We define the representation:
\begin{align} \nonumber
\Phi:
 G & \longrightarrow \Aut(k[[t]])\\
 \label{simpleConj}
\sigma & \longmapsto 
\left(
t \mapsto t(1+\bar{C}_s(\sigma) t^m)^{-1/m}
\right).
\end{align}
The expression $1+\bar{C}_s(\sigma) t^m)^{-1/m}$ can be expanded as a powerseries  using the binomial theorem and determines uniquely an automorphisms of $k[[t]]$. 
We have that for all $\sigma,\tau\in G$
\[
\Phi(\tau\sigma)=\Phi(\sigma)\Phi(\tau).
\]
Furthermore $\Phi$ is a monomorphism. 
\end{theorem}
\begin{proof}
We begin by noticing that $\sigma(\bar{f}_s)=\bar{f}_s+\bar{C}_s(\sigma)$ and  we can  select $t$ so that $t^{-m}=\bar{f}_s$. 
Using the above expression we can determine the value of $\sigma(t)$ using 
\[
\frac{1}{\sigma(t)^{m}}=\frac{1}{t^m}+\bar{C}_s(\sigma), 
\]
see also \cite[eq. 4]{kontogeorgis2008ramification}. 
In this way $\sigma$ coincides
with the image of $\Phi(\sigma)\in \Aut(k[[t]])$ in eq. (\ref{simpleConj}).

Recall that $\sigma \in G$ acts on the elements $\bar{f}_0,\ldots,\bar{f}_{s-1}$ by definition in terms of the cocycles $\bar{C}_i(\sigma)$. 
This was defined to be a left action.
Also this action is by construction assumed to be compatible with the action of $G$ on $k[[t]]$ in the sense that when we see the elements $\bar{f}_i$ as elements in $k[[t]][t^{-1}]$, then $\sigma(\bar{f}_i)=\Phi(\sigma)(\bar{f}_i)$, that is the action of $\sigma$ on $\bar{f}_i$ as elements in $k_{\mathbf{n},m_{i+1}}[\bar{f}_0,\bar{f}_1,\ldots,\bar{f}_{s-1}]$  coincides with the action of $\sigma$ on $f_i$ seen as an element in the quotient field of $k[[t]]$.
In other words we have 
\[
\sigma(f_i(t))=f_i(\sigma(t))=f_i(t)+C_i(\sigma).
\]

We will prove first that this is a homomorphism i.e.
\begin{equation}
\label{hom-notti}
 t(1+ \bar{C}_s(\tau\sigma) t^m))^{-1/m}=
t(1+\bar{C}_s(\sigma)t^m)^{-1/m} \circ 
t(1+\bar{C}_s(\tau)t^m)^{-1/m},
\end{equation}
where $\circ$ denotes the composition of two powerseries. 
The right hand side of the above equation equals
\[
t
\big(1+\bar{C}_s(\tau)t^m
\big)^{\frac{-1}{m}}
\left(
1+ 
\frac
{
\tau(\bar{C}_s(\sigma))
t^m
}
{1+\bar{C}_s(\tau)t^m)}   
\right)^{\frac{-1}{m}} \\
=
t 
\left(
1 +\big(\bar{C}_s(\tau)+\tau\bar{C}_s(\sigma) \big)t^m
\right)^{-1/m}
\]
so eq. (\ref{hom-notti}) holds by the cocycle condition for 
$\bar{C}_s$. 

The kernel of the homomorphism $\Phi$, consists of all elements
$\sigma \in G$ such that $\bar{C}_s(\sigma)=0$. 
But if $\bar{C}_s(\sigma)=0$ then $\sigma(t)=t$ and $\sigma$ is the identity. 
\end{proof}
\begin{remark}
The above construction behaves well when we substitute $f_m$ with $f_m'=f_m+a$. In any case the representation given in eq. (\ref{simpleConj}) is given in terms of the canonical uniformizer 
$t_{f_r}$ corresponding to the element $\bar{f}_s=f_r$ which gives rise to the cocycle $\bar{C}_s$. 
\end{remark}
\begin{remark}
Equation (\ref{simpleConj}) implies that the knowledge
of the cocycle $\bar{C}_s$ implies the knowledge of $\sigma(t)$,
which in turn gives us how $\sigma$ acts on all other elements
$\bar{f}_i$ for all $0\leq i \leq s-1$.
 This seems to imply that $\bar{C}_s$
can determine 
all other cocycles $\bar{C}_\nu$ for all $1 \leq \nu \leq s-1$. 
This is not entirely correct.
Indeed,  $\bar{C}_s$ is a cocycle with values on the $G$-module
$k_{\mathbf{n},\bar{m}_{s}}[\bar{f}_0,\bar{f}_1,\ldots,\bar{f}_{s-1}]$, therefore the action of $G$ on $\bar{f}_{i}$ for $0\leq i \leq s-1$ is assumed to be known 
 and is part of the definition of the cocycle $\bar{C}_s$. This means that $\bar{C}_i$ are assumed to be known and part of the definition of $\bar{C}_{s}$.
\end{remark}

\begin{proposition}
\label{ram-comp}
If $\sigma \in G$, 
$\sigma\neq 1$, then 
\[
v_P(\sigma(t)-t)=m-v_P\big(\bar{C}_s(\sigma)\big)+1=I(\sigma), 
\]
where 
 $-I(\sigma)$ 
is the Artin character since $k$ is algebraically closed, see \cite[VI.2]{serre2013local}.
Therefore $\sigma\in G_{I(\sigma)}-G_{I(\sigma)+1}$. 
\end{proposition}
\begin{proof}
The valuation of $\sigma(t)-t$ comes from the binomial expansion of eq. (\ref{simpleConj}). The rest is the definition of the ramification group. 
\end{proof}
%
%
%
\subsection{Application: Elements of order $\mathbf{p^h}$ in the Nottingham group}
\subsubsection{On the form of elements of order $p$}
It is known that every element of order $p$ in $\Aut(k[[t]])$ is conjugate to the automorphism 
\[
t \mapsto t (1+c t^m)^{-1/m}, \quad \text{ where } c\in k, 
\]
for some $m$ prime to $p$, see \cite[prop. 1.2]{klopsch2000automorphisms}
and \cite[th. 2.2]{lubin2011torsion}.

We can obtain this result using theorem \ref{explicitForm}. 
Let $\sigma$ be an automorphism of $k[[t]]$ of order $p$. Let $X\rightarrow \mathbb{P}^1$ be the corresponding HKG-cover.  
The sequence of higher ramification groups equals $\langle \sigma \rangle=G_0=G_1=\cdots=G_m > \{1\}$, i.e. there is only one jump in the ramification filtration. 
If $m=1$ then $G_{i}(P)=\{1\}$ for $i \geq 2$ and in this case the genus $g_X=0$. This is a trivial case so we can assume that $m>1$.
From theorem \ref{tsouk1} we know that the Weierstrass semigroup is  generated by $p=|G_1(P)|$ and $m_r$. If $m_i$ is a pole number less than $m_r$ then $m_i$ is a multiple of $p$, hence the corresponding elements $f_i$ with pole order $m_i$ at $P$ will be powers of $f_0$ where $(f_0)_\infty=pP$. 


Since the ramification filtration jumps only once, the same holds for the representation filtration, i.e. 
$$G_1(P)=\Ker\rho_{c_1}>\{1\}$$
So if $\sigma$ is not the identity then by \cite[prop.27]{Karanikolopoulos2013} we have that
\begin{align*}
\sigma(f_0^i)&=f_0^i\text{ for }i=0,1,\dots,\lfloor m_r/p\rfloor\text{ and}\\
\sigma(f_{c_1+1})=\sigma(f_r)&=f_r+C(\sigma)\text{ where } C(\sigma)\in k^\times.
\end{align*}
Compare also with the computation of proposition \ref{ram-comp}. To obtain the result we notice the following; changing the local uniformizer to a canonical  one imposes the substitution of $\sigma$ by a conjugate which, by theorem  \ref{explicitForm}, maps
$t$ to the desired form.

%
\subsubsection{Application to the case of cyclic groups}
Let us now consider an element $\sigma$ of order $p^h$. As before the cyclic group 
\[
G_0(P)=G_1(P)=\cdots=G_{b_1}(P)\gneqq 
G_{b_1+1}(P)=\cdots=
G_{b_2}(P)\gneqq\cdots\gneqq G_{b_\mu}(P)\gneqq\{1\}
\]
Since a cyclic group has only cyclic subgroups and all quotients of cyclic groups are cyclic, while 
$G_{b_i}/G_{b_{i+1}}$ is elementary abelian, we see that the number of gaps $\mu$ is equal to $h$
and $p^{h-i}$ is the exact power of $p$ dividing each $\bar{m}_i$.

Observe  that all intermediate elementary abelian extensions 
$F_{i+1}/F_i=F_i(\bar{f}_i)/F_i$ are cyclic. The additive polynomial describing  the extension $F_i(\bar{f}_i)/F_i$ is given by 
\[
Y^p-\B{C}_i^{p-1} Y=\bar{f}_i^p -\bar{C}_i^{p-1} \bar{f}_i,
\]
by computation of the Moore determinant 
$\det
\begin{pmatrix}
C_i & Y \\ C_i^p & Y^p
\end{pmatrix}
$, where $\bar{C}_i$ is computed at a generator 
$\sigma^{p^i}$ of the cyclic group $\mathrm{Gal}(F_{i+1}/F_i)=G_{b_{i+1}}/G_{b_i}$, (i.e. $\sigma^{p^i}(\B{f}_i)=\B{f}_i+\bar{C}_i(\sigma^{p^i})$).
Since $\bar{C}_i\in k$, if we rescale $\bar{f_i}$ by $\bar{f_i}/C_i$, we can 
assume without loss of generality that the equation is an 
Artin-Schreier one:
\[
Y^p-Y=\bar{f}_i^p - \bar{f}_i=D_i, \text{ where } D_i\in F_i.
\] 
Let $g$ be an automorphism of the HKG-cover $X$. 
Since $g(\bar{f}_\nu)=\bar{f}_\nu+\bar{c}_\nu(g)$ and $\bar{c}_\nu(g)\in F_{\nu-1}$, the automorphism $g$ gives rise to an automorphism $g: F_\nu \rightarrow F_\nu$ for all $\nu$. 
We have that 
\begin{equation}
\label{compat*}
\bar{C}_i(g)^p -\bar{C}_i(g) = (g-1)(\bar{f}_i^p -\bar{f}_i)=(g-1)D_i.
\end{equation}
Notice that eq. (\ref{compat*}) has many solutions $\bar{C}_i(g)$ for a fixed $g$, which differ by an element $\bar{c}_i(\sigma)$ for some $\sigma \in \mathrm{Gal{}}(F_{i+1}/F_i)$, since $(g\sigma-1)(D_i)=(g-1)(D_i)$.  

 The representation filtration has the following form (the filtrations are collectively depicted in the diagrams below)
 
\[
F^{G_1(P)}=F_0=F^{\mathrm{ker} \rho_0}
\subset 
F_1=F^{\mathrm{\mathrm{ker} \rho_1}}
\subset \cdots \subset
F_r=F^{\mathrm{ker}\rho_r}=F.
\]
We have $p^{h-i}=|\mathrm{ker} \rho_{c_{i+1}}|$ for $0\leq i \leq n-1$ and $p^{h}=|G_1(P)|$.
The generators of the Weierstrass semigroup are 
$
p^{h},p^{h-1} \lambda_1,\ldots, p\lambda_{\mu-1},\lambda_\mu.
$
We have the following tower of fields:

\noindent
\begin{minipage}[l]{0.38\textwidth}
\[
\xymatrix{
  F=F_{h+1}=F_{h}(\bar{f}_{h}) \ar@{-}[d]^{\mathbb{Z}/p\Z} 
\ar@{-}@/_5pc/_{\Z/p^h\Z}[dddd]
  \\
  F_{h}=F_{h-1}(\bar{f}_{h-1}) \ar@{-}[d]^{\Z/p\Z} \\
  F_{h-1}  \ar@{.}[d] \\
F_2=F_1(\bar{f}_1)\ar@{-}[d]^{\mathbb{Z}/p\Z}\\
  F_1=k(\bar{f}_0) \ar@{=}[d] \\
  F_0=F^{G_1(P)}
}
\]
\end{minipage}
\begin{minipage}[c]{0.3\textwidth}
\[
\xymatrix{
\{1\} \ar@{-}[d]^p\\
G_{b_{h}}=\ker\rho_{c_{h}} \ar@{-}[d]^p\\
G_{b_{h-1}}=\ker\rho_{c_{h-1}}\ar@{.}[d]\\
G_{b_2}=\ker\rho_{c_2}\ar@{-}[d]^{p}\\
G_{b_1}=\ker\rho_{c_1}\ar@{=}[d]\\
G_1(P)
}
\]
\end{minipage}
\begin{minipage}[r]{0.3\textwidth}
\[
\xymatrix{
1\ar@{-}[d]^{p}\\
\langle\sigma^{p^{h-1}}\rangle\,(\text{order=}p)\ar@{-}[d]^p \\
\langle\sigma^{p^{h-2}}\rangle\,(\text{order=}p^2)\ar@{.}[d]\\
\langle\sigma^p\rangle\,(\text{order=}p^{h-1})\ar@{-}[d]^{p}\\
\langle\sigma\rangle \ar@{=}[d]\\
G_1(P)
}
\]
\end{minipage}

For every $g\in \mathrm{Gal}(F/F_1)$ we have
\[
g(\bar{f}_{r-1})-\bar{f}_{r-1}=\bar{C}_{r-1}(g).
\]


For a cyclic group $\Z/p^i\Z$ the cohomology is given by:
\[
H^1(\Z/p^i\Z,A)=\frac{
  \{a \in A: N(a)=0\}
}
{
  (\sigma_i-1)A
},
\]
where $\sigma_i$ is a generator of the cyclic group $\Z/p^i \Z$ and $N=1+\sigma+\dots+\sigma^{p^i-1}$ is the norm, see \cite[th. 6.2.2, p. 168]{Weibel}. 
In view of theorem \ref{kernelCoho} we will consider 
the groups $\mathrm{Gal}(F_{i+1}/F_1)$, which are generated by 
the generator $\sigma$ of the cyclic group $\mathrm{Gal}(F_{h+1}/F_1)$ modulo the subgroup $\mathrm{Gal}(F_{h+1}/F_{i+1})$. 
Thus in the group $\mathrm{Gal}(F_{i+1}/F_1)$ the order of $\sigma$ equals $p^i$.

Observe now that $\tau=\sigma^{p^{i-1}}$ acts trivially on 
$A=k_{\mathbf{n},m_i}[\bar{f}_0,\bar{f}_1,\ldots,\bar{f}_{i-1}]$. We now  compute the norm for $\mathrm{Gal}(F_{i+1}/F_1)$: 
\begin{align*}
1+\sigma+\cdots+\sigma^{p^i-1}
&=
\sum_{\nu=0}^{p^i-1}
\sigma^\nu
=\sum_{\pi=0}^{p-1}
\sum_{\upsilon=0}^{p^{i-1}-1}
\sigma^{\pi p^{i-1}}\sigma^{\upsilon}
\\
&= \sum_{\pi=0}^{p-1} \tau^{\pi}
\sum_{\upsilon=0}^{p^{i-1}-1} \sigma^{\upsilon},
\end{align*}
where $\tau:=\sigma^{p^{i-1}}$,
and observe that the above equation restricted on $A$ gives
\[
1+\sigma+\cdots+\sigma^{p^i-1}=p \cdot 
\sum_{\upsilon=0}^{p^{i-1}-1} \sigma^{\upsilon}
\]
which is zero on $A$. So we finally arrive at the computation:
\[
H^1
\left(
\Z/p^i\Z,
k_{\mathbf{n},\bar{m}_i}[\bar{f}_0,\bar{f}_1,\ldots,\bar{f}_{i-1}]
\right)
=
k_{\mathbf{n},\bar{m}_i}[\bar{f}_0,\bar{f}_1,\ldots,\bar{f}_{i-1}]_{
  \Z/p^i\Z
},
\]
where the latter space is the space of $\Z/p^i\Z$-coinvariants.
\begin{proposition}
\label{22prop}
A cyclic group of the Nottingham group is described by a series of elements $\bar{C}_i\in k_{\mathbf{n},\bar{m}_i}[\bar{f}_0,\ldots,\bar{f}_{i-1}]_{\Z/p^i\Z}$ so that $\bar{C}_i^p-\bar{C}_i$ is zero in the space $k_{\mathbf{n},\bar{m}_i}[\bar{f}_0,\ldots,\bar{f}_{i-1}]_{\Z/p^i\Z}$.

In order to ensure that the element $\sigma$ has order $p^h$
we should have, $\bar{C}_s(\sigma^{p^\nu})\neq 0$, for all $0\leq \nu < h$ i.e. 
\[
\left(
1+\sigma+\cdots+\sigma^{p^{\nu}-1}
\right)\bar{C}(\sigma)\neq 0. 
\]
\end{proposition}


\end{document}